\documentclass{tac}

\usepackage{amscd,amssymb,amsmath,amsfonts,mathrsfs}
\usepackage{graphicx}
\usepackage{verbatim,enumerate,paralist}
\usepackage{textcomp}
\usepackage[colorlinks]{hyperref}
\hypersetup{%
  urlcolor=blue,linkcolor=blue,citecolor=blue
}
\usepackage{pdfpages}

\usepackage[T1]{fontenc}
\usepackage[all,poly,knot,tips]{xy}
\newcommand{\lra}{\longrightarrow}
\newcommand{\lla}{\longleftarrow}
\newcommand{\Lra}{\Longrightarrow}

\newcommand{\ldual}[1]{\mathord{{\let\nolimits\relax\sideset{^\wedge}{}{#1}}}}
\newcommand{\laction}[2]{\mathord{{\let\nolimits\relax\sideset{^{#1}}{}{#2}}}}
\newcommand{\conj}[2]{\mathord{{\let\nolimits\relax\sideset{^{#1}}{}{#2}}}}

\def\CA{{\mathscr A}}

\def\CC{{\mathscr C}}

\def\CI{{\mathscr I}}

\def\CM{{\mathscr M}}

\def\CV{{\mathscr V}}

\def\CX{{\mathscr X}}

\begin{document}

\title{Skew monoidales, skew warpings and quantum categories}

\author{Stephen Lack and Ross Street}

\address{
Department of Mathematics, Macquarie University NSW 2109,
Australia}

\eaddress{steve.lack@mq.edu.au, ross.street@mq.edu.au}

\copyrightyear{2012}

\thanks{Both authors gratefully acknowledge the support of 
Australian Research Council Discovery Grant DP1094883; Lack
acknowledges with equal gratitude the support of an Australian
Research Council Future Fellowship.}



\maketitle


\begin{abstract}
\noindent Kornel Szlach\'anyi \cite{Szl2012} recently used the term skew-monoidal category for a particular laxified version of monoidal category. 
He showed that bialgebroids $H$ with base ring $R$ could be characterized in terms of skew-monoidal structures on the category of one-sided $R$-modules for which the lax unit was $R$ itself. 
We define skew monoidales (or skew pseudo-monoids) in any monoidal bicategory $\CM$. 
These are skew-monoidal categories when $\CM$ is $\mathrm{Cat}$. 
Our main results are presented at the level of monoidal bicategories. 
However, a consequence is that quantum categories \cite{QCat} with base comonoid $C$ in a suitably complete braided monoidal category $\CV$ are precisely skew monoidales in $\mathrm{Comod} (\CV)$ with unit coming from the counit of $C$. 
Quantum groupoids (in the sense of \cite{ChLaSt} rather than \cite{QCat}) are those skew monoidales with invertible associativity constraint.
In fact, we provide some very general results connecting opmonoidal monads and skew monoidales. 
We use a lax version of the concept of warping defined in \cite{Tanduoidal} to modify monoidal structures.   

\end{abstract}

\tableofcontents

\section{Introduction}\label{Intro}

To prove coherence for monoidal categories, Mac Lane \cite{Rice} found that five axioms sufficed. Kelly \cite{Kelly1964} reduced these to two axioms. However, the reduction depends critically on invertibility of the associativity and unit constraints. 

A skew-monoidal category in the sense of Szlach\'anyi \cite{Szl2012} is defined in the same way as a monoidal category 
(see \cite{EilKel1966} or \cite{CWM}) except that the associativity and unit constraints need not be invertible. 
Therefore, the directions of these constraints must be specified. 
For a {\em left skew-monoidal category} $\CC$ with tensor product functor 
$\ast : \CC \times \CC \lra \CC$ and unit object $J$, 
the natural families of lax constraints have the directions 
\begin{equation}
\alpha_{XYZ} : (X\ast Y)\ast Z \lra X\ast (Y \ast Z)
\end{equation}
\begin{equation}
\lambda_X : J\ast X \lra X
\end{equation}
\begin{equation}
\rho_X :  X \lra X\ast J
\end{equation}
subject to five conditions: the pentagon for $\alpha$, a condition relating $\alpha_{XJY}$, 
$\lambda_Y$ and $\rho_X$, one relating $\alpha_{JXY}$ and $\lambda$, 
one relating $\alpha_{XYJ}$ and $\rho$, and one relating $\lambda_J$ and $\rho_J$. 
For a  {\em right skew-monoidal category}, the constraints have their directions reversed. 
A left skew-monoidal structure on $\CC$ yields a right one both on $\CC^{\mathrm{op}}$ (in which morphisms are reversed)  
and on $\CC^{\mathrm{rev}}$ (in which the tensor is switched).  

A left skew monoidal category will be called {\em Hopf} when the associativity constraint is invertible. 
The reason for this term should become clear.
It is called {\em left} 
[{\em right}] {\em normal} 
when the left [right] unit constraints are invertible. 
So a monoidal category is precisely a left and right normal Hopf left skew monoidal category.

It was shown in \cite{Szl2003} that the bialgebroids (as considered by \cite{Takeuchi1977, Lu1996, Xu2001, KaSzl2003}) with base ring $R$ amount to cocontinuous opmonoidal monads on the category of two-sided $R$-modules. This was one motivation for \cite{QCat} in defining quantum categories in a braided monoidal category $\CV$ as monoidal comonads in an appropriate place. When $\CV$ is the dual of the category of abelian groups, quantum categories are bialgebroids.   

What is shown in \cite{Szl2012} is that such bialgebroids are closed left skew-monoidal structures on the category of left $R$-modules. This is the motivation for the present paper where we show that quantum categories are skew-monoidal objects, with a certain unit, in an appropriate monoidal bicategory. The Hopf skew-monoidal objects are quantum groupoids.

We present our main theorems in Section~\ref{2BT} at the level of monoidal bicategories. 
We also relate the work to the fusion operators of \cite{Fusion} and the warpings of monoidal structures of \cite{Tanduoidal}. 

\section{Fusion operators, tricocycloids and bimonoids}\label{Fotab}

Let $\CV$ be a braided monoidal category \cite{BTC}. We write as if it were strict monoidal.

A {\em lax fusion operator} on an object $A$ in $\CV$  is a morphism 
$$V:A\otimes A \lra A\otimes A$$ satisfying the fusion equation
\begin{equation}\label{fusioneq}
V_{23}V_{12} = V_{12} V_{13} V_{23}
\end{equation}
where $V_{12}=V\otimes 1_A$, $V_{23}=1_A\otimes V$ and $V_{13}=(c_{A,A}\otimes1_A)(1_A\otimes V)(c_{A,A}\otimes1_A)^{-1}$.
It is a {\em fusion operator} (as defined in \cite{Fusion}) when $V$ is invertible. 

If we put $v=c_{A,A}V$, the proof of Proposition 1.1 of \cite{Fusion} shows that $V$ satisfies \eqref{fusioneq} if and only if $v$ satisfies the 3-cocycle condition 
\begin{equation}\label{3cocycle}
(v\otimes 1_A)(1_A\otimes c_{A,A})(v\otimes 1_A) = (1_A\otimes v)(v\otimes 1_A)(1_A\otimes v) \ .
\end{equation}

A {\em lax tricocycloid} in $\CV$ is defined to be an object $A$ equipped with a morphism $v:A\otimes A \lra A\otimes A$ satisfying \eqref{3cocycle}. 
The lax tricocycloid $(A,v)$ is {\em augmented} when it is equipped with a unit $\eta : I \lra A$ and a counit $\varepsilon : A \lra I$ satisfying the following four conditions:
\begin{equation}
(A\stackrel{1_A\otimes \eta}\lra A\otimes A \stackrel{v}\lra A\otimes A \stackrel{1_A\otimes \varepsilon}\lra A) = (A\stackrel{1_A}\lra A)
\end{equation}
\begin{equation}
(A\otimes A \stackrel{v}\lra A\otimes A \stackrel{\varepsilon \otimes 1_A}\lra A) = (A\otimes A\stackrel{1_A\otimes \varepsilon}\lra A)
\end{equation}
\begin{equation}
(A\stackrel{\eta \otimes 1_A}\lra A\otimes A \stackrel{v}\lra A\otimes A) = (A\stackrel{1_A \otimes \eta}\lra A\otimes A)
\end{equation}
\begin{equation}
(I\stackrel{\eta}\lra A\stackrel{\varepsilon}\lra I)=(I\stackrel{1_I}\lra I) \ .
\end{equation}

\begin{proposition}
Let $\CV$ be a braided monoidal category and let $A$ be an augmented lax tricocycloid. A left skew monoidal structure on $\CV$ with the same unit object $I$ is defined as follows. 
The new tensor product is 
$$X\ast Y = A\otimes X \otimes Y \ ,$$
the associativity constraint $\alpha_{XYZ} : (X\ast Y)\ast Z \lra X\ast (Y\ast Z)$ is
$$\alpha_{XYZ}  = (1_A \otimes c_{A,X} \otimes 1_Y\otimes 1_Z) (v\otimes 1_X \otimes 1_Y \otimes 1_Z) \ ,$$
the left unit constraint $\lambda_X : I\ast X\lra X$ is 
$$\varepsilon \otimes 1_X : A\otimes X \lra X \  , $$
and the right unit constraint $\rho_X : X\lra X \ast I$ is
$$\eta \otimes 1_X : X  \lra A\otimes X \ .$$ 
\end{proposition}
\begin{proof} 
The proof of Proposition 2.1 of \cite{Fusion} shows that $\alpha$ satisfies the pentagon. The other axioms follow, one by one, from the four axioms on an augmentation. 
\end{proof}

\begin{remark} 
In the terminology to be introduced in Section~\ref{Smon}, 
an augmented lax tricocycloid is precisely a left skew 
monoidal structure on the unit monoid $I$, in the monoidal bicategory 
$\mathrm{Mod}(\CV)$ of monoids in $\CV$ and two-sided modules between them, 
where the unit is $I$ acting on itself on both sides. However, we will now see that it is also something perhaps more familiar.  
\end{remark}  

\begin{proposition}
Let $\CV$ be a braided monoidal category. 
There is a bijection between bimonoid structures on an object $A$, with respect to the inverse braiding, and augmented lax tricocycloid structures on $A$. The bijection is defined by 
$$v=(A\otimes A \stackrel{\delta \otimes 1_A} \lra A\otimes A\otimes A 
\stackrel{1_A\otimes \mu} 
\lra A\otimes A \stackrel{c_{A,A}} \lra A\otimes A)$$ 
and the inverse is defined by
$$\mu=(A\otimes A \stackrel{v} \lra A\otimes A \stackrel{1_A\otimes \varepsilon} \lra A)$$
$$\delta = (A\stackrel{1_A \otimes \eta} \lra A\otimes A \stackrel{v} \lra A\otimes A \stackrel{c_{A,A}^{-1}} \lra A\otimes A) \ ,$$
where, of course, the unit and counit of the bimonoid give the augmentation of the lax tricocycloid.
The bimonoid is Hopf if and only if the operator $v$ of the corresponding lax tricocycloid is invertible. 
\end{proposition}
\begin{proof}
This is quite an enjoyable exercise with the string diagrams of \cite{GTC}. Part of it is done in the proof of Proposition 1.2 of \cite{Fusion}. The last sentence of the Proposition, with string diagrammatic proof, can be found as Proposition 10 of \cite{Flocks}. 
\end{proof}

\section{Skew warpings}\label{Swar}

The concept of warping for a monoidal category was introduced at the end of \cite{Tanduoidal} as a way of modifying the monoidal structure. We now adapt that idea to skew monoidal categories. A possible connection between warpings and skew monoidal categories was mentioned already by Szlach\'anyi \cite{Szl2012}. 

Let $\CA$ denote a left skew monoidal category. A {\em skew left warping} of $\CA$ consists of the following data:
\begin{itemize}
\item[(a)] a functor $T:\CA \lra \CA$;
\item[(b)] an object $K$ of $\CA$; 
\item[(c)] a natural family of morphisms $v_{A,B} : T(TA\otimes B)\lra TA\otimes TB$ in $\CA$;
\item[(d)] a morphism $v_0 : TK \lra I$; and,
\item[(e)] a natural family of morphisms $k_A : A \lra TA\otimes K$;
\end{itemize}
such that the following five diagrams commute.
\begin{equation}\label{warpassoc}
\begin{aligned}
\xymatrix{
T(TA\otimes B)\otimes TC \ar[rr]^-{v_{A,B}\otimes 1} && (TA\otimes TB)\otimes TC \ar[d]^-{\alpha_{TA,TB,TC}} \\
T(T(TA\otimes B)\otimes C) \ar[u]^-{v_{TA\otimes B , C}} \ar[d]_-{T(v_{A,B}\otimes 1)} && TA\otimes (TB\otimes TC) \\
T((TA\otimes TB)\otimes C) \ar[dr]_-{T\alpha_{TA,TB,C}\phantom{AA}} && TA\otimes T(TB\otimes C) \ar[u]_-{1\otimes v_{B,C}} \\
& T(TA\otimes (TB\otimes C)) \ar[ru]_-{\phantom{AA}v_{A,TB\otimes C}} &}
\end{aligned}
\end{equation}
\begin{equation}\label{warpunit1}
\begin{aligned}
\xymatrix{
& TK\otimes TB \ar[rd]^-{ v_0 \otimes 1_{TB}}  & \\
T(TK\otimes B) \ar[ru]^-{v_{K,B}} \ar[d]_-{T(v_0\otimes 1_B)} & & I\otimes TB \ar[d]^-{\lambda_{TB}} \\
T(I\otimes B) \ar[rr]_-{T\lambda_B} & & TB }
\end{aligned}
\end{equation}
\begin{equation}\label{warpunit2}
\begin{aligned}
\xymatrix{
T(TA\otimes K)  \ar[rr]^-{v_{A,K}} && TA\otimes TK \ar[d]^-{1\otimes v_{0}} \\
TA \ar[u]^-{Tk_{A}}  \ar[rr]_{\rho_{TA}} && TA\otimes I 
}\end{aligned}
\end{equation}
\begin{equation}\label{warpunit3}
\begin{aligned}
\xymatrix{
T(TA\otimes B)\otimes K  \ar[rr]^-{v_{A,B}\otimes 1_K} && (TA\otimes TB)\otimes K \ar[d]^-{\alpha_{TA,TB,K}} \\
TA\otimes B \ar[u]^-{k_{TA\otimes B}}  \ar[rr]_{1_{TA}\otimes k_B} && TA\otimes (TB\otimes K) }
\end{aligned}
\end{equation}
\begin{equation}\label{warpunit4}
\begin{aligned}
\xymatrix{
TK\otimes K \ar[rr]^-{v_0\otimes1_K}  && I\otimes K \ar[d]^-{\lambda_K} \\
K \ar[rr]_-{1_K} \ar[u]^-{k_K} && K}
\end{aligned}
\end{equation}

A skew left warping is called {\em Hopf} when each $v_{A,B}$ is invertible. 

\begin{example}\label{auglaxtoswarp}
An augmented lax tricocycloid $(A,v,\eta,\varepsilon)$ in a braided monoidal category $\CV$ determines a skew left warping on $\CV$ with $TX=A\otimes X$, $K=I$, 
$$v_{X,Y} = (A\otimes A \otimes X \otimes Y \stackrel{v\otimes 1_X \otimes 1_Y} \lra A\otimes A \otimes X \otimes Y \stackrel{1_A\otimes c_{A,X} \otimes 1_Y}\lra A\otimes X \otimes A \otimes Y) \ ,$$
$$v_0 = \varepsilon : A\lra I,~\quad\textnormal{and}\quad~k_X = \eta \otimes 1_X : X \lra A\otimes X.$$
\end{example}

\begin{example}\label{DualgivesSWarp}
Suppose $\CA$ is a monoidal category and $K$ is an object in it which has a right dual $R$ with unit $\eta : I \lra R\otimes K$ and counit $\varepsilon : K\otimes R \lra I$. 
A Hopf skew left warping on $\CA$ is defined by 
$TA = A\otimes R$, 
$K = K$, 
$v_{A,B} = 1_{A\otimes R \otimes B \otimes R} $,
$v_0 = \varepsilon$ and $k_A = 1_A \otimes \eta$.   
\end{example}

The definitions of monoidal functor, opmonoidal functor, monoidal natural transformation and opmonoidal natural transformation can be made verbatim for left skew monoidal categories. For example, if $\CA$ and $\CX$ are left skew monoidal categories, an {\em opmonoidal functor} $T:\CA \lra \CX$ is equipped with a natural family of morphisms $\psi_{A,B} : T(A\otimes B) \lra TA\otimes TB$ and a morphism $\psi_0 : TI \lra I$ satisfying the usual three axioms but keeping in mind that the constraints in $\CA$ and $\CX$ are not necessarily invertible. The concept of opmonoidal monad on a left skew-monoidal category is therefore clear. 

\begin{remark}\label{rev} 
A monoidal category $\CA$ is both left and right skew-monoidal by taking the inverse constraints. 
An opmonoidal monad $T$ on $\CA$ is the same as an opmonoidal monad on the monoidal category $\CA^{\mathrm{rev}}$ (which is $\CA$ with reversed tensor product).   
\end{remark}
 
\begin{example}\label{warpfromopmonmonad}
Suppose $T$ is an opmonoidal monad on a left skew monoidal category $\CA$. 
A skew left warping on $\CA$ is defined by 
$T = T$, 
$K = I$, 
$$v_{A,B} = (T(TA\otimes B) \stackrel{\psi_{TA,B}}\lra TTA\otimes TB\stackrel{\mu_A\otimes 1_{TB}}\lra TA\otimes TB) \ , $$
$$v_0 = \psi_0 : TI\lra I$$ and 
$$k_A = (A\stackrel{\eta_A}\lra TA \stackrel{\rho_{TA}} \lra TA\otimes I) \ .$$
This example has a converse as expressed by Proposition~\ref{opmonmonadfromwarp} below.   
\end{example}

The next Proposition contains Remark 2.7 of \cite{BLV2011} concerning the special case where $\CA$ is monoidal; also see \cite{Sttalk}.

\begin{proposition}\label{opmonmonadfromwarp}
Suppose the left skew monoidal category $\CA$ is right normal. Then the construction of Example~\ref{warpfromopmonmonad} determines a bijection between opmonoidal monads on $\CA$ and skew left warpings on $\CA$ for which $K=I$. 
\end{proposition}
\begin{proof}
Given a skew left warping with $K=I$, we obtain a monad structure on $T$ by taking the multiplication $\mu$ to be the composite 
$$TTA \stackrel{T\rho_{TA}}\lra T(TA\otimes I)\stackrel{v_{A,I}}\lra TA\otimes TI 
\stackrel{1_{TA}\otimes v_0}\lra TA\otimes I\stackrel{\rho_{TA}^{-1}}\lra TA$$
and the unit $\eta_A$ to be 
$$A\stackrel{k_A}\lra TA\otimes I\stackrel{\rho_{TA}^{-1}}\lra TA \ .$$
For the opmonoidal structure, we take $\psi_{A,B}$ to be
$$T(A\otimes B)\stackrel{T(\eta_A\otimes 1_B)}\lra T(TA\otimes B)\stackrel{v_{A,B}}\lra TA\otimes TB$$ 
and $\psi_0$ to be $v_0:TI\lra I$. 
\end{proof}

The next result is easily proved.

\begin{proposition}\label{newskew}
A skew left warping of a left skew monoidal category $\CA$ determines another left skew monoidal structure on $\CA$ as follows:
\begin{itemize}
\item[(a)] tensor product functor $A\ast B = TA\otimes B$;
\item[(b)] unit $K$;
\item[(c)] associativity constraint $$T(TA\otimes B)\otimes C \stackrel{v_{A,B}\otimes 1_C}\lra (TA\otimes TB) \otimes C \stackrel{\alpha_{TA,TB,C}}\lra TA\otimes (TB\otimes C) \ ;$$
\item[(d)] left unit constraint $$TK\otimes B\stackrel{v_0\otimes 1_B}\lra I\otimes B \stackrel{\lambda_B}\lra B \ ;$$
\item[(e)] right unit constraint $$A\stackrel{k_A} \lra TA\otimes K \ .$$
\end{itemize}
There is an opmonoidal functor $(T, v_0 , v_{A,B}) : (\CA , \ast , K) \lra (\CA , \otimes , I)$. 
\end{proposition}

\begin{example}\label{newfromdual} 
Applying Proposition~\ref{newskew} to Example~\ref{DualgivesSWarp}, we obtain a Hopf left skew-monoidal structure on any monoidal category $\CA$ from a duality $K\dashv R$. 
The new tensor product is defined by $A\ast B = A\otimes R\otimes B$. We denote $\CA$ equipped with this left skew-monoidal structure by $\CA_K$. 
A bicategorical version of this simple idea will be of some interest below.   
\end{example}

\begin{example}\label{revopmonmonad} 
By Remark~\ref{rev}, we see that an opmonoidal monad on a monoidal category $\CA$ gives rise to both a left and right skew-monoidal structure on $\CA$; the left tensor product of $A$ and $B$ is $TA\otimes B$ while the right is $A\otimes TB$; the unit is $I$ as in $\CA$ in both cases. 
\end{example}

\section{Skew monoidales}\label{Smon}

Let $\CM$ be a monoidal bicategory \cite{MbHa}, which is the 
natural context in which monoidales (or pseudomonoids) can 
be studied \cite{GPS, MbHa, Lack2000}. 
Mostly, we write as if it were a Gray monoid \cite{GPS, MbHa}.

A {\em left skew-monoidal} structure on an object $C$ in $\CM$ 
consists of morphisms \hbox{$p:C\otimes C \lra C$}, $j:I\lra C$, respectively called the {\em tensor product} and {\em unit}, and 2-cells 
\begin{equation}\label{smonoidale1}
\begin{aligned}
\xymatrix{
& C\otimes (C\otimes C) \ar[rr]^-{1\otimes p}_<<<<<<<<<<<{\phantom{AA}}="1" && C\otimes C \ar[dr]^-{p}& \\ 
(C\otimes C)\otimes C \ar[ru]^-{a_{C,C,C}} \ar[rr]_-{p\otimes 1} \ar@{}[rrrr]^<<<<<<<<<<<<<<<<<<<<<<<<<<<<<<<<<<<<<<<<<<<{\phantom{AA}}="2" \ar@{<=}"1";"2"^-{\phantom{a}\alpha} && C\otimes C \ar[rr]_-{p} && C\,\,  \\ }
\end{aligned}
\end{equation}
\begin{equation}\label{smonoidale2}
\begin{aligned}
\xymatrix{
C\otimes C \ar[rd]_{p}^(0.5){\phantom{aa}}="1"   && I\otimes C \ar[ll]_{j\otimes 1}  \ar[ld]^{\ell}_(0.5){\phantom{aa}}="2" \ar@{=>}"1";"2"^-{\lambda}
\\
& C }
\end{aligned}
\end{equation}
\begin{equation}\label{smonoidale3}
\begin{aligned}
\xymatrix{
C \ar[d]_{1}^(0.5){\phantom{aaaaa}}="1" \ar[rr]^{r}  && C\otimes I \ar[d]^{1\otimes j}_(0.5){\phantom{aaaaa}}="2" \ar@{=>}"1";"2"^-{\rho}
\\
C  && C\otimes C \ar[ll]^-{p} & , }
\end{aligned}
\end{equation}
respectively called the {\em associativity}, {\em left unit} and {\em right unit constraints}, satisfying the following five conditions.
\begin{equation}\label{smonoidaleax1}
\begin{aligned}
\xymatrix{
& p(p\otimes 1)(1\otimes p\otimes 1)  \ar[rd]^-{\phantom{A}\alpha (1\otimes p\otimes 1)}  & \\
p(p\otimes 1)(p\otimes 1\otimes 1) \ar[ru]^-{p(\alpha \otimes 1)\phantom{A}} \ar[d]_-{\alpha(p\otimes 1 \otimes 1)} & & p(1\otimes p)(1\otimes p\otimes 1) \ar[d]^-{p(1\otimes \alpha)} \\
p(1\otimes p)(p\otimes 1\otimes 1) \ar[rd]_-{\cong} & & p(1\otimes p)(1\otimes 1\otimes p)  \\
& p(p\otimes 1)(1\otimes 1\otimes p)  \ar[ru]_-{\alpha(1\otimes 1 \otimes p)} }
\end{aligned}
\end{equation}

\begin{equation}\label{smonoidaleax2}
\begin{aligned}
\xymatrix{
p(p\otimes 1)(1\otimes j \otimes 1) \ar[rr]^-{\alpha (1\otimes j \otimes 1)}  && p(1\otimes p)(1\otimes j \otimes 1) \ar[d]^-{p(1\otimes \lambda)} \\
p \ar[rr]_-{1} \ar[u]^-{p(\rho \otimes 1)} && p}
\end{aligned}
\end{equation}

\begin{equation}\label{smonoidaleax3}
\begin{aligned}
\xymatrix{
p(p\otimes 1)(j\otimes 1 \otimes 1) \ar[rr]^-{p(\lambda \otimes 1)} \ar[d]_-{\alpha (j\otimes 1 \otimes 1) } && p \\
p(1\otimes p)(j\otimes 1 \otimes 1) \ar[rr]_-{\cong} && p(j\otimes 1) p  \ar[u]_-{\lambda p} }
\end{aligned}
\end{equation}

\begin{equation}\label{smonoidaleax4}
\begin{aligned}
\xymatrix{
p \ar[rr]^-{p(1 \otimes \rho)} \ar[d]_-{\rho p} && p(1\otimes p)(1\otimes 1 \otimes j)  \\
p(1\otimes j) p \ar[rr]_-{\cong} && p(p\otimes 1)(1\otimes 1 \otimes j) \ar[u]_-{\alpha (1\otimes 1 \otimes j) } }
\end{aligned}
\end{equation}

\begin{equation}\label{smonoidaleax5}
\begin{aligned}
\xymatrix{
p(1\otimes j)j \ar[rr]^-{\cong}  && p(j\otimes 1)j \ar[d]^-{\lambda j} \\
j \ar[rr]_-{1} \ar[u]^-{\rho j} && j}
\end{aligned}
\end{equation}

An object $C$ of $\CM$ equipped with a left skew-monoidal structure is called a {\em left skew monoidale} in $\CM$.

\begin{remark}\label{SMonBicat}
The microcosm principle applies here. There is a concept of {\em left skew-monoidal bicategory}. Indeed, the definition of tricategory in \cite{GPS} has the constraints precisely in the required directions. What we want here is the one object case. So a left skew-monoidal structure on a bicategory $\CM$ has a pseudofunctor $\otimes : \CM \times \CM \lra \CM$ and unit object $I$ with pseudonatural constraints
$$
a_{XYZ} : (X\otimes Y)\otimes Z \lra X\otimes (Y \otimes Z)
$$
$$
\ell_X : I\otimes X \lra X
$$
$$
r_X :  X\lra X\otimes I \ .
$$
Where there were five axioms for a left skew-monoidal category, there are now higher-order constraints which we take to be invertible modifications in the diagrams for those axioms. There are axioms on these modifications essentially as set out in \cite{GPS}. (There is presumably an even more skew version of monoidal bicategory where these modifications are not required to be invertible but as yet we have no need for that generality.)

The point we wish to make is that {\em left skew monoidale} makes sense in any left skew-monoidal bicategory. Indeed, the reason for writing the structures \eqref{smonoidale1}, \eqref{smonoidale2} and \eqref{smonoidale3} the way we have is to show this level of generality. However, the axioms \eqref{smonoidaleax1} to \eqref{smonoidaleax5} must be redrawn more fully using pasting diagrams since they currently are appropriate only for a Gray monoid. 

To complete this remark, notice that {\em right skew monoidale} also makes sense in any left skew-monoidal bicategory. This is because reversing 2-cells in $\CM$ yields a left skew-monoidal structure on $\CM^\mathrm{co}$; then a right skew monoidale in $\CM$ can be defined to be a left skew monoidale in $\CM^\mathrm{co}$.   
   
\end{remark}

\begin{example}
A left skew monoidale in the cartesian monoidal 2-category $\mathrm{Cat}$ of categories, functors and natural transformations is a left skew-monoidal category as in Section~\ref{Intro}. 
\end{example}

\begin{example}\label{lsVCat}
More generally, for braided monoidal category $\CV$, a left skew monoidale $\CC$ in the monoidal 2-category $\CV$-$\mathrm{Cat}$ of $\CV$-categories, $\CV$-functors and $\CV$-natural transformations is defined to be a left skew-monoidal $\CV$-category.
\end{example}

\begin{example}\label{lsConv}
For a symmetric closed monoidal category $\CV$ which is complete and cocomplete, a left skew monoidale $\CC$ in the monoidal bicategory $\CV$-$\mathrm{Mod}$ of $\CV$-categories, two-sided $\CV$-modules and $\CV$-module morphisms is defined to be a {\em left skew-promonoidal} $\CV${\em -category}. As for the case of a promonoidal $\CV$-category in Brian Day's doctoral thesis (see \cite{DayConv}), the same convolution formulas give a closed left skew-monoidal $\CV$-category $[\CC^{\mathrm{op}},\CV]$.    
\end{example}

\begin{example}
As a particular case of Example~\ref{lsConv} (actually, with far fewer conditions on $\CV$ and care with the braiding), we see from Section~\ref{Fotab} that each bimonoid in $\CV$ yields a left skew-promonoidal structure on the unit $\CV$-category $\CI$.    
\end{example}

It is possible to lift most of what we have said about skew-monoidal categories to skew monoidales.
In particular, the notion of {\em skew left warping} on a monoidale $A$ makes sense: just adapt in the obvious way the definition of warping on $A$ as given in Section 8 of \cite{Tanduoidal}; the data are:
\begin{itemize}
\item[(a)] a morphism $t: A \lra A$;
\item[(b)] a morphism $k: I\lra A$;
\item[(c)] a $2$-cell
$$\xymatrix{
& A\otimes A \ar[rr]^-{p}_<<<<<<<<<<<{\phantom{A}}="1" && A \ar[dr]^-{\phantom{A}t}& \\ 
A\otimes A \ar[ru]^-{t\otimes 1} \ar[rr]_-{t\otimes t} \ar@{}[rrrr]^<<<<<<<<<<<<<<<<<<<<<<<<<<<<<<<<{\phantom{A}}="2" \ar@{=>}"1";"2"^-{\phantom{a}v} && A\otimes A \ar[rr]_-{p} && A\,\, ; \\ 
}$$
\item[(d)] a $2$-cell
$$\xymatrix{
\ar@{}[rrrr]_<<<<<<<<<<<<<<<<<<<<<<*!/d1.5pt/{\labelstyle{\phantom{A}}}="1" && A \ar[rrd]^-{t}  && \\
I \ar[rru]^-{k} \ar[rrrr]_-{j}^-{\phantom{a}}="2" \ar@{=>}"1";"2"^-{\phantom{a}v_{0}} &&&& A\,\,;
}$$
\item[(e)] a $2$-cell
$$\xymatrix{
\ar@{}[rrrr]_<<<<<<<<<<<<<<<<<<<<<<<<<<<*!/d1.5pt/{\labelstyle{\phantom{A}}}="1" && A\otimes A \ar[rrd]^-{p}  && \\
A \ar[rru]^-{t\otimes k} \ar[rrrr]_-{1_A}^-{\phantom{a}}="2" \ar@{<=}"1";"2"^-{\phantom{a}\kappa} &&&& A\,\,;
}$$
\end{itemize}
subject to five axioms. We have more general forms of Propositions~\ref{opmonmonadfromwarp} and \ref{newskew}. 

\begin{proposition}\label{opmonmonad=warp}
Suppose the left skew monoidale $A$ is right normal. Then there is a bijection between opmonoidal monads on $A$ and skew left warpings on $A$ for which $k=j$. 
\end{proposition}

\begin{proposition}\label{gennewskew}
A skew left warping of a left skew monoidale $A$ determines another left skew monoidal structure on $A$ as follows:
\begin{itemize}
\item[(a)] tensor product functor $p_t = p(t\otimes1_A):A\otimes A \lra A$;
\item[(b)] unit $k : I\lra A$;
\item[(c)] associativity constraint $$p(t\otimes 1)(p\otimes 1)(t\otimes 1\otimes 1) \stackrel{p(v\otimes 1)}\Lra p(p\otimes 1) (t\otimes t\otimes 1) \stackrel{\alpha(t\otimes t\otimes 1)}\Lra p(1\otimes p) (t\otimes t\otimes 1) \ ;$$
\item[(d)] left unit constraint $$p(tk\otimes 1)\stackrel{p(v_0\otimes 1)}\Lra p(j\otimes 1) \stackrel{\lambda}\Lra 1_A \ ;$$
\item[(e)] right unit constraint $$1_A\stackrel{\kappa} \Lra p(t\otimes k) \ .$$
\end{itemize}
There is an opmonoidal morphism $(t,v_0,v) : (A,k,p_t) \lra (A,p,j)$.
\end{proposition}

\begin{example}
Suppose $K$ is an object with a right bidual $R$ in the monoidal bicategory 
$\CM$, so that we have a unit $n:I\lra R\otimes K$ and counit $e:K\otimes R\lra I$. 
We can generate the monoidale $K^{\mathrm{e}}=R\otimes K$ with tensor product 
$p=1\otimes e \otimes 1$ and unit $j=n$. 
Consideration of bialgebroids motivates us to consider opmonoidal monads 
$t$ on $K^{\mathrm{e}}$. By Proposition~\ref{opmonmonad=warp}, 
these are in bijection with skew left warpings $t$ on $K^{\mathrm{e}}$ with $k=n$. 
Now by Proposition~\ref{gennewskew}, any such determines a new left skew monoidal structure $p_t$ on $K^{\mathrm{e}}$. 
\end{example}

\section{Two bicategorical theorems}\label{2BT}

Let $\CM$ be a monoidal bicategory containing an object $K$ with a right bidual $R$. 
A Hopf left skew-monoidal bicategory $\CM_K$ is defined by bumping 
Example~\ref{newfromdual} up a dimension. 
The tensor product is $A\ast B = A\otimes R\otimes B$ with unit $K$. 
Let $K^{\mathrm{e}}$ denote the monoidale $R\otimes K$ with tensor product 
$$p=1\otimes e\otimes 1: R\otimes K\otimes R\otimes K\lra R\otimes K$$ and unit 
$$j=n:I\lra R\otimes K \ .$$

Our purpose in this section is to give two reinterpretations of opmonoidal monads on $K^{\mathrm{e}}$ as right skew monoidales. The first is quite general while the second requires a monadicity condition satisfied in examples of interest. 

The biduality $K\dashv_\mathrm{b} R$ gives an equivalence of categories
\begin{equation}\label{rawequiv1}
\CM (K^{\mathrm{e}},K^{\mathrm{e}}) \simeq \CM (K\ast K,K) 
\end{equation}
taking $t: R\otimes K \lra R\otimes K$ to $\hat{t}: K\otimes R\otimes K\stackrel{1\otimes t}\lra K\otimes R\otimes K\stackrel{e\otimes 1}\lra K$. 
Similarly, we have equivalences of categories
\begin{equation}\label{rawequiv2}
\CM (K^{\mathrm{e}}\otimes K^{\mathrm{e}},K^{\mathrm{e}}) \simeq \CM (K\ast K\ast K,K) 
\end{equation}
and
\begin{equation}\label{rawequiv3}
\CM (I,K^{\mathrm{e}}) \simeq \CM (K,K) \ . 
\end{equation}

\begin{theorem}\label{opmonmonadsasSkew1}
The equivalence \eqref{rawequiv1} induces an equivalence of categories between opmonoidal monads $t$ on $K^{\mathrm{e}}$ in the monoidal bicategory $\CM$ and right skew-monoidal structures on $K$ with unit $1_K$ in the Hopf left skew-monoidal bicategory $\CM_K$.   
\end{theorem}
\begin{proof}
By Proposition~\ref{opmonmonad=warp}, an opmonoidal monad $t$ on $K^{\mathrm{e}}$ 
can equally be thought of as a skew right warping of $K^{\mathrm{e}}$ with unit $k=n$. 
So we have a right fusion 2-cell 
$$v:t(1\otimes e\otimes 1)(1_{K^{\mathrm{e}}}\otimes t)\Lra (1\otimes e\otimes 1)(t\otimes t):K^{\mathrm{e}}\otimes K^{\mathrm{e}}\lra K^{\mathrm{e}} \ ,$$ 
and 2-cells 
$$v_0= \psi_0 : tn\Lra n:I\lra K^{\mathrm{e}}$$ and 
$$\kappa = \eta :1_{K^{\mathrm{e}}} \Lra (1\otimes e\otimes 1)(n\otimes t)\cong t: K^{\mathrm{e}}\lra K^{\mathrm{e}} \ ,$$ 
satisfying five axioms.  Now look at the definition of right skew-monoidale: see \eqref{smonoidale1}, \eqref{smonoidale2} and \eqref{smonoidale3} with $\alpha$, $\lambda$ and $\rho$ reversed. 
In our case, with $p=\hat{t}$ and $j=1_K$, we see that we require 
$$\alpha_K : \hat{t}(1_K\otimes 1_R\otimes \hat{t})\Lra \hat{t}(\hat{t}\otimes 1_R\otimes 1_K):  K\ast K\ast K \lra K \ ,$$
$$\rho_K: \hat{t}(1_K\otimes n) \Lra 1_K :K \lra K$$ and
$$\lambda_K : e\otimes 1_K \Lra \hat{t} : K\ast K\lra K $$
satisfying the five axioms. 
We obtain the equivalence of the Theorem by choosing $v$, $v_0$ and $\kappa$ to correspond respectively to $\alpha_K$, $\rho_K$ and $\lambda_K$ under the equivalences \eqref{rawequiv2}, \eqref{rawequiv3} and \eqref{rawequiv1}.  
\end{proof}

Now we present what seems to us to be the deepest result of the paper. 

\begin{theorem}\label{opmonmonadsasSkew2} 
Suppose $\CM$ is a monoidal bicategory in which composition on both sides with a morphism preserves any reflexive coequalizers that exist in the hom categories.
Suppose $K\dashv_\mathrm{b} R$ is a biduality in $\CM$. 
Suppose also that $j:I\lra R$ is an opmonadic (=~Kleisli-type) morphism in $\CM$ and that the opmonadicity is preserved by tensoring with objects on both sides. 
Let $j^*:R\lra I$ be a right adjoint for $j$ and put $j_{\circ}$ equal to $(I\stackrel{n}\lra R\otimes K\stackrel{j^*\otimes 1}\lra K)$.
There is an equivalence of categories between opmonoidal monads $t$ on $K^{\mathrm{e}}$ in the monoidal bicategory $\CM$ and right skew-monoidal structures on $K$ with unit $j_{\circ}$ in the monoidal bicategory $\CM$. 
The skew tensor product $q_t$ for $K$ corresponding to $t$ is the composite 
$$K\otimes K\stackrel{1\otimes j\otimes 1}\lra K\otimes R\otimes K\stackrel{1\otimes t}\lra K\otimes R\otimes K\stackrel{e\otimes 1}\lra K \ .$$   
\end{theorem}
\begin{proof}
Let $s=j^*j$ be the monad generated by the adjunction $j\dashv j^*$.
We shall use Theorem~\ref{opmonmonadsasSkew1} to avoid the opmonoidal monad $t$ by dealing with the data $\hat{t}$, $1_K$, $\alpha_K$, $\rho_K$ and $\lambda_K$ which already form a right skew-monoidal structure on $K$ somewhere, namely, in $\CM_K$. 
While $q_t$ is obtained from $\hat{t}$ simply as the composite 
$K\otimes K\stackrel{1\otimes j\otimes 1}\lra K\otimes R\otimes K\stackrel{\hat{t}}\lra K$, 
in order to use the opmonadic property of $1\otimes j\otimes 1:K\otimes K \lra K\otimes R\otimes K$ 
to reproduce a $\hat{t}$ from a right skew tensor morphism $q:K\otimes K\lra K$, 
we require an opaction 
$$\zeta : q(1\otimes s\otimes 1)\Lra q:K\otimes K\lra K \ .$$  
This 2-cell is obtained as follows. Using the adjunction $j\otimes 1\dashv j^*\otimes 1$, we obtain a mate $\nu : j\otimes 1\Lra (1\otimes q)(n\otimes 1)$ for $\lambda : 1\Lra q(j_{\circ}\otimes 1):K\lra K$. We define $\zeta$ to be the composite 
\begin{eqnarray*}
q(1\otimes j^*\otimes 1)(1\otimes j\otimes 1)\stackrel{q(1\otimes j^*\otimes 1)(1\otimes \nu)}\Lra q(1\otimes j^*\otimes 1)(1\otimes 1\otimes q)(1\otimes n\otimes 1)
\\ \cong q(1\otimes q)(1\otimes j^*\otimes 1\otimes 1)(1\otimes n\otimes 1) \stackrel{\alpha (1\otimes j^*\otimes 1\otimes 1)(1\otimes n\otimes 1)} \Lra 
\\ q(q\otimes 1)(1\otimes j^*\otimes 1\otimes 1)(1\otimes n\otimes 1)\stackrel{q(\rho \otimes 1)}\Lra q \ .
\end{eqnarray*}  
The opaction axioms and the necessity for $\zeta$ to be defined this way can be verified.

Similarly, the constraint $\alpha : q(1\otimes q)\Lra q(q\otimes 1)$ is obtained simply from $\alpha_K$ by precomposing with 
$1\otimes j\otimes 1\otimes j\otimes 1: K\otimes K\otimes K\lra K\otimes R\otimes K\otimes R\otimes K$. 
However, to reproduce $\alpha_K$ from $\alpha$, we use the opmonadicity of 
$1\otimes j\otimes 1\otimes 1\otimes 1$ and  $1\otimes 1\otimes 1\otimes j\otimes 1$; 
it requires verifying that $\alpha$ is compatible with the coactions coming from $\zeta$.

A right adjoint $j^{\circ}$ for $j_{\circ}$ is defined by the composite 
$$K\otimes K\stackrel{1\otimes j}\lra K\otimes R \stackrel{e}\lra I \ .$$
So $\lambda : 1\Lra q_t(j_{\circ}\otimes 1)$ and $\lambda_K(1\otimes j\otimes 1) : j^{\circ}\otimes 1 \Lra q_t$ are mates. Again, to retrieve $\lambda_K$ from $\lambda$, the opmonadicity of $1\otimes j\otimes 1$ is used.     

Since $j\otimes 1_K \cong n j^{\circ} : K \lra R\otimes K$, we see that $\rho_K(1\otimes j^{\circ}):q\Lra 1\otimes j^{\circ}$ and $\rho : q(1\otimes j_{\circ})\Lra1$ are mates. To retrieve $\rho_K$ from $\rho$ we need to use the fact that the counit $\varepsilon : jj^*\Lra 1_R$ is the coequalizer of the two 2-cells $(\varepsilon jj^*), (jj^*\varepsilon) : jj^*jj^*\Lra jj^*$.    \end{proof}

\section{Quantum categories}\label{Qc}

Let $\CV$ be a braided monoidal category which has equalizers of coreflexive pairs such that tensoring $V\otimes -$ with a fixed object $V$ preserves those equalizers. 
This allows the construction (see \cite{D&A}, \cite{QCat} or \cite{Chikh2011}) of the autonomous monoidal bicategory 
$\mathrm{Comod}(\CV)$ whose objects are comonoids $C$ in $\CV$ and whose morphisms 
$M:C\lra D$ are left $C$-, right $D$-comodules. 
Recall that the composite $NM=M\otimes_D N$ of comodules $M:C\lra D$ and $N:D\lra E$ is given by the usual coreflexive equalizer.  
The right bidual of $C$ is written $C^{\circ}$ with counit $e:C\otimes C^{\circ}\lra I$ and unit $n:I \lra C^{\circ}\otimes C$. 
There is a canonical monoidale (= pseudomonoid) structure on $C^{\mathrm{e}} = C^{\circ}\otimes C$ with multiplication 
$p=1\otimes e \otimes 1:C^{\mathrm{e}}\otimes C^{\mathrm{e}} \lra C^{\mathrm{e}}$ 
and unit $j=n:I\lra C^{\mathrm{e}}$. 

A \textit{quantum category} $(C,A)$ \textit{in} $\CV$ is a monoidal comonad $A$ on the canonical monoidale $C^{\mathrm{e}} = C^{\circ}\otimes C$. This means that, in $\mathrm{Comod}(\CV)$, we have a morphism  $A:C^{\mathrm{e}}\lra C^{\mathrm{e}}$ equipped with 2-cells 
\begin{equation}
\varepsilon : A \Lra 1_{C^{\mathrm{e}}}
\end{equation}
and 
\begin{equation}
\delta : A \Lra A\otimes_{C^{\mathrm{e}}} A \ ,
\end{equation}
satisfying the three comonad axioms, plus two more 2-cells
\begin{equation}
\begin{aligned}
\xymatrix{
C^{\mathrm{e}} \otimes C^{\mathrm{e}} \ar[d]_{A\otimes A}^(0.5){\phantom{AAAAA}}="1" \ar[rr]^{p}  && C^{\mathrm{e}} \ar[d]^{A}_(0.5){\phantom{AAAAA}}="2" \ar@{=>}"1";"2"^-{\phi_2}
\\
C^{\mathrm{e}} \otimes C^{\mathrm{e}} \ar[rr]_-{p} && C^{\mathrm{e}} }
\end{aligned}
\end{equation}  
and
\begin{equation}
\begin{aligned}
\xymatrix{
I \ar[rd]_{j}^(0.5){\phantom{A}}="1" \ar[rr]^{j}  && C^{\mathrm{e}} \ar[ld]^{A}_(0.5){\phantom{A}}="2" \ar@{=>}"1";"2"_-{\phi_0}
\\
& C^{\mathrm{e}}  & , }
\end{aligned} 
\end{equation}
satisfying the three axioms making $A$ a monoidal morphism and the four axioms making $\varepsilon$ and $\delta$ monoidal 2-cells; this makes ten axioms in all.   

\begin{example} 
A quantum category $(C,A)$ with $C=I$ is precisely a bimonoid $A$ in $\CV$: the comultiplication is $\delta$, the counit is $\varepsilon$, the multiplication is $\phi_2$, and the unit is $\phi_0$, while the ten axioms are exactly those for a bimonoid $A$.    
\end{example} 
\begin{example} 
The quantum category $(C,A)$ with $A$ the identity two-sided comodule $C^{\mathrm{e}}$ is called the {\em chaotic quantum category} on the comonoid $C$.    
\end{example} 
\begin{example} 
Recall from \cite{QCat} that, when $\CV$ is the cartesian monoidal category $\mathrm{Set}$, 
the bicategory $\mathrm{Comod}(\CV)$ is B\'enabou's bicategory $\mathrm{Span}$ of spans \cite{Ben1967} 
made monoidal using the cartesian product of sets.  
It was pointed out in \cite{QCat} that a quantum category $(C,A)$ here is an ordinary (small) category: 
$C$ is the set of objects and $A$ is the span 
$$C\times C \stackrel{(s,t)}\lla A \stackrel{(s,t)}\lra C\times C$$ 
where $s$ and $t$ are the source and target functions. 
Of course, as pointed out in \cite{Ben1967}, a category is more simply a monad $A$ on $C$ in the mere bicategory $\mathrm{Span}$: 
for this $C$ and $A$ are as before as sets but $A$ is the span 
$$C\stackrel{s}\lla A \stackrel{t}\lra C \ .$$ 
However, important for us here is the point of view that categories are also left skew monoidales $(C,A)$ in the monoidal bicategory $\mathrm{Span}$: this time $C$ and $A$ are as before as sets but the tensor product $A$ on $C$ is the span 
$$C\times C \stackrel{(s,t)}\lla A \stackrel{t}\lra C \ .$$         
\end{example} 

Given a morphism $A:C^{\mathrm{e}}\lra C^{\mathrm{e}}$ in $\mathrm{Comod}(\CV)$, we write $A^{\prime}$ for the composite
\begin{equation}
A^{\prime}=(C\stackrel{\varepsilon^{\ast} \otimes 1}\lra C^{\circ}\otimes C \stackrel{A}\lra C^{\circ}\otimes C)
\end{equation}
(where $\varepsilon$ is the counit of the comonoid $C$ and $\varepsilon^{\ast}$ is the right adjoint comodule it determines) 
and then $\bar{A}$ for the composite
\begin{equation}
\bar{A}=(C\otimes C \stackrel{1\otimes A^{\prime}} \lra C\otimes C^{\circ} \otimes C \stackrel{e\otimes 1} \lra C) \ . 
\end{equation}
Notice that, as objects of $\CV$, the comodules $A$, $A^{\prime}$ and $\bar{A}$ are equal. 

\begin{theorem}
There is a bijection between quantum category structures in $\CV$ on $(C,A)$ and left skew-monoidal structures on $C$, with tensor product morphism $\bar{A}$ and unit morphism $\varepsilon^{\ast}$, in $\mathrm{Comod}(\CV)$ using the inverse braiding. 
\end{theorem}
\begin{proof}
This is an application of Theorem~\ref{opmonmonadsasSkew2} with 
$$\CM = \mathrm{Comod}(\CV)^{\mathrm{co}} \ .$$
The biduality $C\dashv_{\mathrm{b}}C^{\circ}$ in 
$\mathrm{Comod}(\CV)$ remains as such in $\CM$ and the monoidal comonad 
$A$ on the monoidale $C^{\mathrm{e}}$ in $\mathrm{Comod}(\CV)$ becomes an opmonoidal monad on 
$C^{\mathrm{e}}$ in $\CM$. 
The counit $\varepsilon_{\circ}$ of the comonoid $C^{\circ}$ yields a coopmonadic morphism 
$\varepsilon_{\circ}^* : I\lra C^{\circ}$ in $\mathrm{Comod}(\CV)$. This gives our monadic $j=\varepsilon_{\circ}^* : I\lra C^{\circ}$ in $\CM$ to which we can apply Theorem~\ref{opmonmonadsasSkew2}. The result follows on noting that right skew monoidales in $\CM$ 
are left skew monoidales in $\mathrm{Comod}(\CV)$.     

An outline of the two constructions is as follows.
Start with a quantum category $(C,A)$ in $\CV$. Then we have a monoidal comonad $A: C^{\mathrm{e}} \lra C^{\mathrm{e}}$. The left fusion map \cite{BLV2011, ChLaSt, dlVP} for $A$ is 
\begin{equation}
v^{\ell} = (p\cdot (A\otimes A) \stackrel{p\cdot(1_A\otimes \delta)} \Lra p\cdot (A\otimes A)\cdot (1\otimes A)  \stackrel{\phi_2 \cdot (1\otimes A)} \Lra A \cdot p \cdot (1\otimes A)) \ .
\end{equation}
From the calculations:
\begin{eqnarray*}
\lefteqn{\bar{A}\cdot (\bar{A}\otimes 1_C)}  \\
& \cong & (e\otimes 1)(1 \otimes A)(1 \otimes \varepsilon^{\ast} \otimes 1)(e\otimes 1)(1 \otimes A \otimes 1)(1\otimes \varepsilon^{\ast} \otimes 1 \otimes 1) \\
& \cong & (e\otimes 1)(1 \otimes 1 \otimes e \otimes 1)(1\otimes 1\otimes 1\otimes A)(1 \otimes 1 \otimes 1\otimes \varepsilon^{\ast} \otimes 1) \\ 
&& (1 \otimes A \otimes 1) (1\otimes \varepsilon^{\ast} \otimes 1 \otimes 1) \\
& \cong & (e\otimes 1)(1 \otimes 1 \otimes e \otimes 1)(1\otimes A \otimes A)(1 \otimes 1 \otimes 1\otimes \varepsilon^{\ast} \otimes 1) \\ 
&& (1\otimes \varepsilon^{\ast} \otimes 1 \otimes 1) \\
& \cong & (e\otimes 1)(1 \otimes p)(1\otimes A \otimes A)(1\otimes \varepsilon^{\ast} \otimes 1 \otimes 1\otimes 1) (1 \otimes 1 \otimes \varepsilon^{\ast} \otimes 1)
\end{eqnarray*}
and
\begin{eqnarray*}
\lefteqn{\bar{A}\cdot (1_C \otimes \bar{A})}  \\
& \cong & (e\otimes 1)(1 \otimes A)(1 \otimes \varepsilon^{\ast} \otimes 1)(1\otimes e\otimes 1)(1 \otimes 1 \otimes A)(1\otimes 1 \otimes  \varepsilon^{\ast} \otimes 1) \\
& \cong &  (e\otimes 1)(1 \otimes A)(1\otimes 1\otimes e\otimes 1)(1 \otimes 1 \otimes 1\otimes A) (1\otimes \varepsilon^{\ast} \otimes 1\otimes 1\otimes 1)  \\
&& (1 \otimes 1 \otimes \varepsilon^{\ast} \otimes 1) \\ 
& \cong & (e\otimes 1)(1 \otimes A)(1\otimes p)(1 \otimes 1 \otimes 1\otimes A) (1\otimes \varepsilon^{\ast} \otimes 1\otimes 1\otimes 1)  \\
&& (1 \otimes 1 \otimes \varepsilon^{\ast} \otimes 1)
\end{eqnarray*}
we see that we can whisker the left fusion 2-cell $v^{\ell}$ to define a 2-cell
\begin{equation}
\alpha : \bar{A} \cdot (\bar{A}\otimes 1_C) \Lra {\bar{A}\cdot (1_C \otimes \bar{A})}
\end{equation}
which is our associativity constraint. The right unit constraint 
\begin{equation}
\rho : 1_C \Lra \bar{A} \cdot (1_C \otimes \varepsilon^{\ast})
\end{equation}
is taken to be $\eta:C\lra A$ as a morphism of comodules from $C$ to $C$.
The left unit constraint 
\begin{equation}
\lambda : \bar{A} \cdot (\varepsilon^{\ast} \otimes 1_C) \Lra 1_C
\end{equation}
is taken to be $t:A\lra C$ using the notation of \cite{QCat, Chikh2011}.

Now assume $(C,M,\varepsilon^{\ast}, \alpha,\lambda,\rho)$ is a left skew monoidale in $\mathrm{Comod}(\CV)$. 
So we have $M:C\otimes C \lra C$, $\alpha : M\cdot (M\otimes 1_C) \Lra M\cdot (1_C \otimes M)$, $\lambda :M_1\Lra 1_C$, $\rho :1_C \Lra  M_2$ where 
$$ M_1 = (C\stackrel{\varepsilon^{\ast}\otimes1}\lra C\otimes C\stackrel{M} \lra C)$$ 
and
$$ M_2 = (C\stackrel{1_C\otimes \varepsilon^{\ast}}\lra C\otimes C\stackrel{M} \lra C) \ .$$
As objects of $\CV$, we have $M_1=M_2=M$. 
As objects of $\CV$, we also have the {\em composable pair object} 
$$M\cdot(M\otimes 1) \cong M_2\otimes_C M_2=:P$$ and the {\em cospan object}  
$$M\cdot(1\otimes M) \cong M_2\otimes_C M_1=:\tilde{P} \ .$$       
We define 
$$\delta:=(M\stackrel{\rho \otimes 1} \lra P \stackrel{\alpha} \lra \tilde{P} \stackrel{\mathrm{equ}} \lra M\otimes M)$$
(where $\mathrm{equ}$ is the defining equalizer for composition in $\mathrm{Comod}(\CV)$) which makes $M$ a comonoid with counit
$$\varepsilon : = (M\stackrel{\lambda} \lra C \stackrel{\varepsilon} \lra I) \ .$$
The coaction of $M_1$ gives
$$r:=(M \stackrel{\delta} \lra M\otimes C^{\circ}\otimes C \stackrel{\varepsilon \otimes 1\otimes 1} \lra C^{\circ}\otimes C)$$
which turns out to be a comonoid morphism. 
The adjunction $r_{\ast} \dashv r^{\ast}$ then generates our comonad 
$$A : C^{\circ}\otimes C \lra C^{\circ}\otimes C$$
on $C^{\mathrm{e}}$ in $\mathrm{Comod}(\CV)$. 
Notice that $A = M$ as objects of $\CV$.
The monoidal structure on $A$ is defined by
$$\mu : = (P \stackrel{\alpha} \lra \tilde{P} \stackrel{\mathrm{equ}} \lra M\otimes M \stackrel{1\otimes \varepsilon} \lra M)$$ and
$$\eta : = (C\stackrel{\rho}\lra M) \ .$$

\end{proof}


\bibliographystyle{plain}

\end{document}